\documentclass{amsart}%
\usepackage{amsfonts}
\usepackage{amsmath}
\usepackage{amssymb}
\usepackage{graphicx}
\usepackage{hyperref}%
\providecommand{\U}[1]{\protect\rule{.1in}{.1in}}
\newtheorem{theorem}{Theorem}
\theoremstyle{plain}

\newtheorem{definition}{Definition}

\newtheorem{lemma}{Lemma}

\newtheorem{problem}{Problem}

\newtheorem{remark}{Remark}

\numberwithin{equation}{section}
\begin{document}
\title[ ]{A note on the Choquet type operators}
\author{Sorin G. Gal}
\address{Department of Mathematics and Computer Science\\
University of Oradea\\
University\ Street No. 1, Oradea, 410087, Romania}
\email{galso@uoradea.ro, galsorin23@gmail.com}
\author{Constantin P. Niculescu}
\address{Department of Mathematics, University of Craiova\\
Craiova 200585, Romania}
\email{constantin.p.niculescu@gmail.com}
\thanks{Appears in \emph{Aequat. Math.} Published online April 11, 2021. }
\date{April 12, 2021}

\begin{abstract}
In this note Choquet type operators are introduced in connection with
Choquet's theory of integrability with respect to a not necessarily additive
set function. Based on their properties, a quantitative estimate for the
nonlinear Korovkin type approximation theorem associated to
Bernstein-Kantorovich-Choquet operators is proved. The paper also includes a
large generalization of H\"{o}lder's inequality within the framework of
monotone and sublinear operators acting on spaces of continuous functions.

\end{abstract}
\maketitle

\section{Introduction}

Choquet's theory of integrability (as described by Denneberg \cite{Denn},
Grabisch \cite{Gr2016} and Wang and Klir \cite{WK}) emphasizes the importance
of a new class of nonlinear operators that verify a mix of conditions
characteristic of Choquet's integral. Its technical definition is detailed as follows.

Given a Hausdorff topological space $X,$ we will denote by $\mathcal{F}(X)$
the vector lattice of all real-valued functions defined on $X$ endowed with
the pointwise ordering. Two important vector sublattices of it are
\[
C(X)=\left\{  f\in\mathcal{F}(X):\text{ }f\text{ continuous}\right\}
\]
and
\[
C_{b}(X)=\left\{  f\in\mathcal{F}(X):\text{ }f\text{ continuous and
bounded}\right\}  .
\]
With respect to the sup norm, $C_{b}(X)$ becomes a Banach lattice. See
\cite{Sch1974} for the theory of these spaces.

As is well known, all norms on the $N$-dimensional real vector space
$\mathbb{R}^{N}$ are equivalent. See Bhatia \cite{Bhatia2009}, Theorem 13, p.
16. When endowed with the sup norm and the coordinate-wise ordering,
$\mathbb{R}^{N}$ can be identified (algebraically, isometrically and in order)
with the space $C\left(  \left\{  1,...,N\right\}  \right)  $, where $\left\{
1,...,N\right\}  $ carries the discrete topology.

Suppose that $X$ and $Y$ are two Hausdorff topological spaces and $E$ and $F$
are respectively ordered vector subspaces of $\mathcal{F}(X)$ and
$\mathcal{F}(Y).$ An operator $T:E\rightarrow F$ is said to be a \emph{Choquet
type operator }(respectively a\emph{ Choquet type functional when
}$F=\mathbb{R}$) if it satisfies the following three conditions:

\begin{enumerate}
\item[(Ch1)] (\emph{Sublinearity}) $T$ is subadditive and positively
homogeneous, that is,%
\[
T(f+g)\leq T(f)+T(g)\quad\text{and}\quad T(af)=aT(f)
\]
for all $f,g$ in $E$ and $a\geq0;$

\item[(Ch2)] (\emph{Comonotone additivity}) $T(f+g)=T(f)+T(g)$ whenever the
functions $f,g\in E$ are comonotone in the sense that
\[
(f(s)-f(t))\cdot(g(s)-g(t))\geq0\text{ for all }s,t\in X;
\]

\item[(Ch3)] (\emph{Monotonicity}) $f\leq g$ in $E$ implies $T(f)\leq T(g).$
\end{enumerate}

All the aforementioned conditions are independent of each other.

If a nonlinear operator $T$ is monotone and positively homogeneous then
necessarily%
\[
T(0)=0\text{ and }f\geq0\text{ implies }T(f)\geq0;
\]
the converse works for linear operators but not in the general case.

The Choquet integral associated to a vector capacity with values in
$\mathbb{R}^{N}$ is a natural source of Choquet type operators. See Remark
\ref{rem3}. For more examples (important in approximation theory) see
\cite{GN2020}, where the following extension of Korovkin's approximation
theorem to the framework of Choquet type operators was proved.

\begin{theorem}
\label{thm1} \emph{(The nonlinear extension of Korovkin's theorem: the several
variables case)} Suppose that $X$ is a locally compact subset of the Euclidean
space $\mathbb{R}^{N}$ and $E$ is a vector sublattice of $\mathcal{F}(X)$ that
contains the $2N+2$ test functions $1,~\pm\operatorname*{pr}_{1}%
,...,~\pm\operatorname*{pr}_{N}$ and $\sum_{k=1}^{N}\operatorname*{pr}_{k}%
^{2}$. $($\emph{Here} $\operatorname*{pr}\nolimits_{k}:(x_{1},...,x_{N}%
)\rightarrow x_{k}$ $(k=1,...,N)$ \emph{denote the} \emph{canonical
projections on} $\mathbb{R}^{N})$.

$(i)$ If $(T_{n})_{n}$ is a sequence of monotone and sublinear operators from
$E$ into $E$ such that
\[
\lim_{n\rightarrow\infty}T_{n}(f)=f\text{\quad uniformly on the compact
subsets of }X
\]
for each of the $2N+2$ aforementioned test functions, then the above limit
property also holds for all nonnegative functions $f$ in $E\cap C_{b}(X)$.

$(ii)$ If, in addition, each operator $T_{n}$ is comonotone additive, then
$(T_{n}(f))_{n}$ converges to $f$ uniformly on the compact subsets of $X$, for
every $f\in E\cap C_{b}\left(  X\right)  $.

Notice that in both cases $(i)$ and $(ii)$ the family of testing functions can
be reduced to $1,~-\operatorname*{pr}_{1},...,~-\operatorname*{pr}_{N}$ and
$\sum_{k=1}^{N}\operatorname*{pr}_{k}^{2}$ when $K$ is included in the
positive cone of $\mathbb{R}^{N}$. Also, the convergence of $(T_{n}(f))_{n}$
to $f$ $\ $is uniform on $X$ when $f\in E$ is uniformly continuous and bounded
on $X.$
\end{theorem}

In this paper we prove a quantitative estimate concerning the above
Korovkin-type theorem in the case of Bernstein-Kantorovich-Choquet operators
but our argument works also for the Sz\'{a}sz-Mirakjan-Kan\-torovich-Choquet
operators, the Baskakov-Kan\-to\-ro\-vich-Cho\-quet operators etc. See Theorem
\ref{thm4}, which is based on a generalization of the
Cauchy-Bunyakovsky-Schwarz inequality for Choquet type operators (stated as
Lemma 1).

A large generalization of H\"{o}lder's inequality within the framework of
monotone and sublinear operators acting on spaces of continuous functions
makes the objective of Theorem 3.

For the convenience of the reader, we devoted Section 2 to an overview of
basic facts about monotone capacities and the Choquet integral.

\section{Preliminaries on Choquet's Integral}

Given a nonempty set $X,$ by a \emph{lattice} of subsets of $X$ we mean any
collection $\Sigma$ of subsets that contains $\emptyset$ and $X$ and is closed
under finite intersections and unions. A lattice $\Sigma$ is an \emph{algebra}
if in addition it is closed under complementation. An algebra closed under
countable unions and intersections is called a $\sigma$-algebra.

Of special interest is the case where $X$ is a compact Hausdorff space and
$\Sigma$ is either the lattice $\Sigma_{up}^{+}(X)$ of all upper contour
closed sets $S=\left\{  x\in X:f(x)\geq t\right\}  ,$ or the lattice
$\Sigma_{up}^{-}(X)$ of all upper contour open sets $S=\left\{  x\in
X:f(x)>t\right\}  $ associated to pairs $f\in C(X)$ and $t\in\mathbb{R}.$

When $X$ is a compact metrizable space, $\Sigma_{up}^{+}(X)$ coincides with
the lattice of all closed subsets of $X$ (and $\Sigma_{up}^{-}(X)$ coincides
with the lattice of all open subsets of $X$).

In what follows $\Sigma$ denotes a lattice of subsets of an abstract set $X$.

\begin{definition}
\label{defcap}A set function $\mu:\Sigma\rightarrow\lbrack0,\infty)$ is called
a capacity if it verifies the following two conditions:

$(C1)$ $\mu(\emptyset)=0;$ and

$(C2)~\mu(A)\leq\mu(B)$ for all $A,B\in\Sigma$, with $A\subset B$
\emph{(}monotonicity\emph{)}.

The capacity $\mu$ is called normalized if $\mu(X)=1.$
\end{definition}

If $\Sigma$ is an algebra of subsets of $X,$ then to every capacity $\mu$
defined on $\Sigma$, one can attach a new capacity $\overline{\mu}$, the
\emph{dual} of $\mu,$ which is defined by the formula
\[
\overline{\mu}(A)=\mu(X)-\mu(X\setminus A).
\]
Notice that $\overline{\left(  \bar{\mu}\right)  }=\mu.$

The capacities provide a non additive generalization of \emph{probability
measures}, that is, of capacities $\mu$ having the property of $\sigma
$-additivity,%
\[
\mu\left(
%TCIMACRO{\dbigcup \nolimits_{n=1}^{\infty}}%
%BeginExpansion
{\displaystyle\bigcup\nolimits_{n=1}^{\infty}}
%EndExpansion
A_{n}\right)  =%
%TCIMACRO{\dsum \nolimits_{n=1}^{\infty}}%
%BeginExpansion
{\displaystyle\sum\nolimits_{n=1}^{\infty}}
%EndExpansion
\mu(A_{n})
\]
for every sequence $A_{1},A_{2},A_{3},...$ of disjoint sets belonging to
$\Sigma$ such that $\cup_{n=1}^{\infty}A_{n}\in\Sigma.$

Some other classes of capacities exhibiting extensions of the properties of
additivity or $\sigma$-additivity are listed below.

A capacity $\mu$ is called \emph{submodular} (or strongly subadditive) if%
\begin{equation}
\mu(A\cup B)+\mu(A\cap B)\leq\mu(A)+\mu(B)\text{\quad for all }A,B\in\Sigma.
\label{submod}%
\end{equation}
Every additive measure is also submodular, but the converse fails. A
normalized submodular capacity $\mu$ defined on an algebra $\Sigma$ of sets
has the property%
\begin{equation}
\mu(A)=0\text{ implies }\mu(\complement A)=1. \label{eqnullsets}%
\end{equation}

A capacity $\mu$ is called \emph{lower continuous\ }(or continuous by
ascending sequences)\emph{ }if%
\[
\lim_{n\rightarrow\infty}\mu(A_{n})=\mu(%
%TCIMACRO{\dbigcup \nolimits_{n=1}^{\infty}}%
%BeginExpansion
{\displaystyle\bigcup\nolimits_{n=1}^{\infty}}
%EndExpansion
A_{n})
\]
for every nondecreasing sequence $(A_{n})_{n}$ of sets in $\Sigma$ such that
$\cup_{n=1}^{\infty}A_{n}\in\Sigma;$ $\mu$ is called \emph{upper
continuous\ }(or continuous by descending sequences) if $\lim_{n\rightarrow
\infty}\mu(A_{n})=\mu\left(  \cap_{n=1}^{\infty}A_{n}\right)  $ for every
nonincreasing sequence $(A_{n})_{n}$ of sets in $\Sigma$ such that $\cap
_{n=1}^{\infty}A_{n}\in\Sigma.$ If $\mu$ is an additive capacity defined on a
$\sigma$-algebra, then its upper/lower continuity is equivalent to the
property of $\sigma$-additivity.

If $\Sigma$ is a $\sigma$-algebra, then a capacity $\mu:\Sigma\rightarrow
\lbrack0,1]$ is lower (upper continuous) if and only if its dual $\bar{\mu}$
is upper (lower) continuous.

There are several standard procedures to attach to a probability measure
certain not necessarily additive capacities. So is the case of \emph{distorted
probabilities,} $\mu(A)=u(P(A)),$ obtained from a given probability measure
$P:\Sigma\rightarrow\lbrack0,1]$ and applying to it a distortion
$u:[0,1]\rightarrow\lbrack0,1],$ that is, a nondecreasing and continuous
function such that $u(0)=0$ and $u(1)=1.$ For example, one may chose
$u(t)=t^{a}$ with $\alpha>0.$ When the distortion $u$ is concave (for example,
when $u(t)=t^{a}$ with $0<\alpha<1$ or when $u(t)=\frac{2t}{t+1}$), then $\mu$
is an example of lower continuous submodular capacity.

The following concept of integrability with respect to a capacity $\mu
:\Sigma\rightarrow\lbrack0,\infty)$ was introduced by Choquet \cite{Ch1954},
\cite{Ch1986}. It concerns the class of \emph{upper} \emph{measurable}
functions, that is, the functions $f:X\rightarrow\mathbb{R}$ such that all
upper contour sets $\left\{  x\in X:f(x)\geq t\right\}  $ belong to $\Sigma$.

\begin{definition}
\label{def2} The Choquet integral of an upper measurable function $f$ on a set
$A\in\Sigma$ is defined as the sum of two Riemann improper integrals,
\begin{multline*}
(C)\int_{A}f\mathrm{d}\mu\\
=\int_{0}^{+\infty}\mu\left(  \{x\in A:f(x)\geq t\}\right)  \mathrm{d}%
t+\int_{-\infty}^{0}\left[  \mu\left(  \{x\in A:f(x)\geq t\}\right)
-\mu(A)\right]  \mathrm{d}t.
\end{multline*}
Accordingly, $f$ is said to be Choquet integrable if both integrals above are finite.
\end{definition}

Every upper measurable and bounded function is Choquet integrable. If $f\geq
0$, then the last integral in the formula appearing in Definition \ref{def2}
is 0.

When $\Sigma$ is a $\sigma$-algebra, the upper measurability and the Borel
measurability are equivalent and the Choquet integral coincides with the
Lebesgue integral for $\sigma$-additive measures. Besides, the inequality sign
$\geq$ in the above two integrands can be replaced by $>;$ see \cite{WK},
Theorem 11.1, p. 226.

The next remarks summarize the basic properties of the Choquet integral:

\begin{remark}
\label{rem1}$(a)$ If $f$ and $g$ are two upper measurable functions which are
Choquet integrable, then
\begin{gather*}
f\geq0\text{ implies }(\operatorname*{C})\int_{X}f\mathrm{d}\mu\geq0\text{
\quad\emph{(}positivity\emph{)}}\\
f\leq g\text{ implies }\left(  \operatorname*{C}\right)  \int_{X}%
f\mathrm{d}\mu\leq\left(  \operatorname*{C}\right)  \int_{X}g\mathrm{d}%
\mu\text{ \quad\emph{(}monotonicity\emph{)}}\\
\left(  \operatorname*{C}\right)  \int_{X}af\mathrm{d}\mu=a\cdot\left(
\operatorname*{C}\right)  \int_{X}f\mathrm{d}\mu\text{ for all }a\geq0\text{
\quad\emph{(}positive\emph{ }homogeneity\emph{)}}\\
\left(  \operatorname*{C}\right)  \int_{X}1\cdot\mathrm{d}\mu(t)=\mu
(X)\text{\quad\emph{(}calibration\emph{).}}%
\end{gather*}
$(b)$ In general, the Choquet integral is not additive but \emph{(}as was
noticed by Dellacherie \emph{\cite{Del1970})}, if $f$ and $g$ are comonotonic
\emph{(}that is, $(f(\omega)-f(\omega^{\prime}))\cdot(g(\omega)-g(\omega
^{\prime}))\geq0$, for all $\omega,\omega^{\prime}\in X$\emph{), }then
\[
\left(  \operatorname*{C}\right)  \int_{X}(f+g)\mathrm{d}\mu=\left(
\operatorname*{C}\right)  \int_{X}f\mathrm{d}\mu+\left(  \operatorname*{C}%
\right)  \int_{X}g\mathrm{d}\mu.
\]
An immediate consequence is the property of translation invariance,
\[
\left(  \operatorname*{C}\right)  \int_{X}(f+c)\mathrm{d}\mu=\left(
\operatorname*{C}\right)  \int_{X}f\mathrm{d}\mu+c\cdot\mu(X)
\]
for all $c\in\mathbb{R}$ and all Choquet integrable functions $f$ $.$

$(c)$ If $\mu$ is a lower continuous capacity, then the Choquet integral is
lower continuous in the sense that
\[
\lim_{n\rightarrow\infty}\left(  \left(  C\right)  \int_{X}f_{n}\mathrm{d}%
\mu\right)  =\left(  C\right)  \int_{X}f\mathrm{d}\mu,
\]
whenever $(f_{n})_{n}$ is a nondecreasing sequence of bounded random variables
that converges pointwise to the bounded variable $f.$

For $(a)$ and $(b)$, see Denneberg\emph{ \cite{Denn}, }Proposition\emph{ 5.1,
}p\emph{. 64; (c) }follows in a straightforward way from the definition of the
Choquet integral.

$(d)$ If $\mu\leq\nu$ are two capacities, then $(C)\int_{X}f\mathrm{d}\mu
\leq(C)\int_{X}f\mathrm{d}\nu$, for all nonnegative measurable functions $f.$

$(e)$ $(C)\int_{A}-f\mathrm{d}\mu=-(C)\int_{A}f\mathrm{d}\overline{\mu}$. See
\emph{\cite{WK}}, Theorem \ \emph{11.7}, p. \emph{233}.
\end{remark}

\begin{remark}
\label{rem2}\emph{(The Subadditivity Theorem) }If $\mu$ is a submodular
capacity, then\ the associated Choquet integral is subadditive, that is,%
\[
\left(  \operatorname*{C}\right)  \int_{X}(f+g)\mathrm{d}\mu\leq\left(
\operatorname*{C}\right)  \int_{X}f\mathrm{d}\mu+\left(  \operatorname*{C}%
\right)  \int_{X}g\mathrm{d}\mu
\]
for all $f$ and $g$ integrable on $X.$ See\emph{ \cite{Denn}, }Theorem\emph{
6.3, }p\emph{. 75. }In addition, the following two integral analogs of the
modulus inequality hold true,
\[
|(\operatorname*{C})\int_{X}f\mathrm{d}\mu|\leq(\operatorname*{C})\int
_{X}|f|\mathrm{d}\mu
\]
and
\[
|(\operatorname*{C})\int_{X}f\mathrm{d}\mu-(\operatorname*{C})\int
_{X}g\mathrm{d}\mu|\leq(\operatorname*{C})\int_{X}|f-g|\mathrm{d}\mu.
\]
The last assertion is covered by Corollary \emph{6.6}, p. \emph{82}, in
\emph{\cite{Denn}. }
\end{remark}

\begin{remark}
\label{rem2bis}If $\mu$ is a submodular capacity, then\ the associated Choquet
integral is a submodular functional in the sense that
\[
\left(  C\right)  \int_{A}\sup\left\{  f,g\right\}  \mathrm{d}\mu+\left(
C\right)  \int_{A}\inf\{f,g\}\mathrm{d}\mu\leq\left(  C\right)  \int
_{A}f\mathrm{d}\mu+(C)\int_{A}g\mathrm{d}\mu
\]
for all $f$ and $g$ integrable on $X.$ For this, integrate term by term the
inequality
\begin{multline*}
\mu\left(  \left\{  x:\sup\{f,g\}(x)\geq t\right\}  \right)  +\mu\left(
\left\{  x:\inf\{f,g\}(x)\geq t\right\}  \right) \\
\leq\mu\left(  \left\{  x:f(x)\geq t\right\}  \right)  +\mu\left(  \left\{
x:g(x)\geq t\right\}  \right)  .
\end{multline*}

\end{remark}

The Choquet integral associated to any lower continuous capacity is a
comonotonically additive, monotone and lower continuous functional. The
converse also holds.

\begin{theorem}
\label{thm2}Suppose that $X$ is a compact Hausdorff space and
$I:C(X)\rightarrow\mathbb{R}$ is a comonotonically additive and monotone
functional such that $I(1)=1$. Then $I$ is also lower continuous and there
exists a unique lower continuous normalized capacity $\mu:\Sigma_{up}%
^{-}(X)\rightarrow\lbrack0,1]$ such that
\[
I(f)=\int_{0}^{+\infty}\mu\left(  \{x\in X:f(x)>t\}\right)  \mathrm{d}%
t+\int_{-\infty}^{0}\left[  \mu\left(  \{x\in X:f(x)>t\}\right)  -1\right]
\mathrm{d}t
\]
for all $f\in C(X).$ Moreover, if $I$ is submodular in the sense that%
\[
I(\sup\left\{  {f,g}\right\}  )+I(\inf\left\{  {f,g}\right\}  )\leq
I(f)+I(g)\text{\quad for all }f,g\in C(X),
\]
then $\mu$ is submodular too.
\end{theorem}

\begin{proof}
Let $(f_{n})_{n}$ and $f$ in $C(X)$, with $(f_{n})$ nondecreasing and
$\lim_{n\rightarrow\infty}f_{n}(x)=f(x)$, for all $x\in X$. Since $I$ is
monotone, it is immediate that
\[
\lim_{n\rightarrow\infty}I(f_{n})\leq I(f).
\]
On the other hand, choose any arbitrary $\varepsilon>0$ and take
$g=f-\varepsilon1$, that is $f=g+\varepsilon1$. Then, $\lim_{n\rightarrow
\infty}f_{n}(x)=f(x)>g(x)$, for all $x\in X$. Since $X$ is compact and
$(f_{n})$ is a nondecreasing sequence of continuous functions, by Dini's
theorem, there is an integer $N$, such that $f_{n}(x)>g(x)=f(x)-\varepsilon1$,
for all $x\in X$ and $n\geq N$. Taking into account the comonotonic additivity
and monotonicity of $I$, we infer that
\[
I(f_{n})\geq I(f-\varepsilon1)=I(f)-\varepsilon I(1)
\]
for all $n\geq N.$ Passing to the limit, first as $n\rightarrow\infty$ and
next as $\varepsilon\rightarrow0$, we obtain $\lim_{n\rightarrow\infty}%
I(f_{n})\geq I(f).$ Since the other inequality was already noticed, we
conclude that $\lim_{n\rightarrow\infty}I(f_{n})=I(f).$

The integral representation of $I$ is part of a more general result due to
Cerreia-Vioglio et al. See \cite{CMMM2012}, Proposition 17, p. 907. As
concerns the correspondence between the property of submodularity of $I$ and
$\mu$, this follows by adapting the argument in \cite{CMMM2012}, Theorem 13
(c), p. 901.
\end{proof}

A result similar to Theorem \ref{thm2}, but for the comonotonically additive,
monotone and upper continuous functionals, was shown by Zhou \cite{Zhou}.

\begin{remark}
\label{rem3}\emph{(Vector capacities)} The aforementioned theory of
integration with respect to a capacity can be easily extended by considering
vector capacities. A simple example is offered by the set functions
$\boldsymbol{\mu}$ defined on the lattice $\Sigma_{up}^{+}(X)$ $($associated
to a compact Hausdorff space $X)$ and taking values in the positive cone of
$\mathbb{R}^{N}$ in such a way that%
\[
\boldsymbol{\mu}\left(  \emptyset\right)  =0\text{ and }\boldsymbol{\mu
}(A)\leq\boldsymbol{\mu}(B)\text{ if }A\subset B.
\]
The concepts of upper/lower continuity and submodularity extend verbatim to
the case of vector capacities. Moreover, a vector capacity $\boldsymbol{\mu}$
is upper continuous $($lower continuous, submodular etc.$)$ if and only if all
its components $\mu_{k}=\operatorname*{pr}_{k}\circ\boldsymbol{\mu}$ are
scalar capacities in the sense of Definition \ref{def2}, with the respective
property. Therefore, the integral with respect to a submodular vector capacity
$\boldsymbol{\mu},$
\[
(C)\int_{X}f\mathrm{d}\boldsymbol{\mu}=\left(  (C)\int_{X}f\mathrm{d}\mu
_{1},...,(C)\int_{X}f\mathrm{d}\mu_{N}\right)  ,
\]
defines a Choquet type operator from $C(X)$ to $\mathbb{R}^{N}.$

According to Theorem \emph{\ref{thm2}}, this construction generates all
Choquet type operators from $C(X)$ to $\mathbb{R}^{N}.$ More general results
concerning the theory of Choquet type operators taking values in an arbitrary
ordered Banach space are available in \emph{\cite{GN2021}}.
\end{remark}

\section{The extension of H\"{o}lder's inequality}

The extension of H\"{o}lder's inequality to the framework of Choquet integral
was treated by numerous authors, see for example \cite{A2019}, \cite{Cerda},
\cite{MP}. By adapting the standard argument based on Young's inequality (see,
\cite{NP2018}, section 1.2, pp. 11-13), H\"{o}lder's inequality for the range
of parameters $p\in(1,\infty)$ and $1/p+1/q=1$ can be further extended to the
general framework of sublinear and monotone operators. Recall that Young's
inequality for this choice of parameters asserts that for all nonnegative
numbers $u,v$ we have%
\begin{equation}
uv\leq\frac{u^{p}}{p}+\frac{v^{q}}{q}\text{\quad for all }u,v\geq0
\label{eqy1}%
\end{equation}
and the equality occurs if and only if $u^{p}=v^{q}.$

\begin{theorem}
\label{thm3}$($\emph{H\"{o}lder's inequality for} $p\in(1,\infty)$ and
$1/p+1/q=1)$ Suppose that $X$ and $Y$ \textit{are two Hausdorff topological
spaces and }$E$ and $F$ are respectively vector sublattices of $C_{b}%
(X)$\textit{ and }$C_{b}(Y)$ which contain the unit \emph{(}the function
identically $1$\emph{)}. Then every \textit{sublinear and monotone operator}
$T:E\rightarrow F$\textit{ for which }$T(1)=1$\textit{ verifies the
inequality}%
\begin{equation}
T(|fg|)\leq\lbrack T(|f|^{p})]^{1/p}\cdot\lbrack T(|g|^{q})]^{1/q}
\label{eqh1}%
\end{equation}
for all $f,g\in E$ such that $fg\in E.$
\end{theorem}

\begin{proof}
For $y\in Y$ arbitrarily fixed, consider the sublinear and monotone functional
$A_{y}:E\rightarrow\mathbb{R}$ defined by the formula
\[
A_{y}(f)=\left(  T(f)\right)  (y).
\]
Clearly, $A_{y}(1)=1$.

Assuming $A_{y}(|f|^{p})>0$ and $A_{y}(|g|^{q})>0,$ we apply inequality
(\ref{eqy1}) for $u=|f|/A_{y}(|f|^{p})^{1/p}$ and $v=|g|/A_{y}(|g|^{q}%
)^{{1/q}}$ to infer that%
\begin{equation}
\frac{|f|}{A_{y}(|f|^{p})^{1/p}}\frac{|g|}{A_{y}(|g|^{q})^{1/q}}\leq\frac
{1}{p}\cdot\frac{|f|^{p}}{A_{y}(|f|^{p})}+\frac{1}{q}\cdot\frac{|g|^{q}}%
{A_{y}(|g|^{q})}. \label{eqy2}%
\end{equation}
Since the functional $A_{y}$ is monotone and sublinear, the last inequality
implies
\[
\frac{A_{y}(|fg|)}{A_{y}(|f|^{p})^{1/p}\cdot A_{y}(|g|^{q})^{1/q}}\leq\frac
{1}{p}+\frac{1}{q}=1,
\]
that is, $T(|f\cdot g|)(y)\leq\lbrack T(|f|^{p})(y)]^{1/p}\cdot\lbrack
T(|g|^{q})(y)]^{1/q}$, which is inequality (\ref{eqh1}) in the statement.

If $A_{y}(|f|^{p})=0$ and/or $A_{y}(|g|^{q})=0,$ then one repeats the above
reasoning by replacing in (\ref{eqy2}) the vanishing number(s) by an
$\varepsilon>0$ arbitrarily small and then passing to the limit as
$\varepsilon\rightarrow0$ to conclude that $A_{y}(|f|\cdot|g|)=0.$ The proof
is done.
\end{proof}

\begin{remark}
\label{rem4}$($\emph{Conditions for equality in Theorem}~\emph{\ref{thm3}}$)$
We assume that $X$ is a compact Hausdorff space and $T:C(X)\rightarrow C(X)$
is a Choquet type operator such that $T(1)=1$ and
\[
T(\sup\left\{  {f,g}\right\}  )+T(\inf\left\{  {f,g}\right\}  )\leq
T(f)+T(g)\text{ for all }f,g\in C(X);
\]
the last condition is nothing but the property of submodularity.

For $x\in X$ arbitrarily fixed, let us consider the comonotone additive and
monotone functional
\[
A_{x}:C(X)\rightarrow\mathbb{R},\text{\quad}A_{x}(f)=\left(  T(f)\right)
(x).
\]
Clearly, $A_{x}(1)=1$ and $A_{x}$ is a submodular functional. According to
Theorem \emph{\ref{thm2}} there exists a unique normalized, lower-continuous
and submodular capacity $\mu_{x}$ on $\Sigma_{up}^{-}(X)$, such that
$A_{x}(f)=(C)\int_{X}fd\mu_{x}$. In this case,%
\[
(C)\int_{X}|h|\mathrm{d}\mu_{x}=0\text{ is equivalent to }\mu_{x}\left(
\{t\in X:|h(t)|>0\}\right)  =0
\]
whenever $h\in C(X)$. See \emph{\cite{WK}}, Theorem \emph{11.3}, p. \emph{228}.

We have equality in \emph{(\ref{eqh1})} at the point $x$ every time when
$A_{x}(|f|^{p})=0$ and/or $A_{x}(|g|^{q})=0,$ equivalently,
\[
\mu_{x}\left(  \{t\in X:|f(t)|>0\}\right)  =0\text{\quad and/or\quad}\mu
_{x}\left(  \{t\in X:|g(t)|>0\}\right)  =0.
\]
According to \emph{(\ref{eqnullsets}), }this means that equality occurs when%
\[
|f(t)|=0\text{ except for a }\mu_{x}\text{-null set\quad and/or\quad
}|g(t)|=0\text{ except for a }\mu_{x}\text{-null set.}%
\]

Suppose now that $A_{x}(|f|^{p})>0$ and $A_{x}(|g|^{q})>0$. In this case an
inspection of the proof of Theorem \emph{\ref{thm3}} shows that equality
occurs in \emph{(\ref{eqh1})} at the point $x$ if
\[
\mu_{x}\left\{  t\in X:\frac{1}{p}\cdot\frac{|f(t)|^{p}}{A_{x}(|f|^{p})}%
+\frac{1}{q}\cdot\frac{|g(t)|^{q}}{A_{x}(|g|^{q})}>\frac{|f(t)|}{A_{x}%
(|f|^{p})^{1/p}}\frac{|g(t)|}{A_{x}(|g|^{q})^{1/q}}\right\}  =0,
\]
equivalently,
\begin{equation}
\frac{1}{p}\cdot\frac{|f(t)|^{p}}{A_{x}(|f|^{p})}+\frac{1}{q}\cdot
\frac{|g(t)|^{q}}{A_{x}(|g|^{q})}=\frac{|f(t)|}{A_{x}(|f|^{p})^{1/p}}%
\frac{|g(t)|}{A_{x}(|g|^{q})^{1/q}}, \label{Y1}%
\end{equation}
except possibly a $\mu_{x}$-null set. According to the equality case in
Young's inequality, this implies the existence of two positive constants
$\alpha$ and $\beta$ such that
\begin{equation}
\alpha|f(t)|^{p}=\beta|g(t)|^{q} \label{Y2}%
\end{equation}
except possibly a $\mu_{x}$-null set.
\end{remark}

If an operator $T:E\rightarrow F$ is monotone and subadditive, then it
verifies the inequality%
\begin{equation}
\left\vert T(f)-T(g)\right\vert \leq T\left(  \left\vert f-g\right\vert
\right)  \text{\quad for all }f,g. \label{eqmod}%
\end{equation}
Indeed, $f\leq g+\left\vert f-g\right\vert ~$yields $T(f)\leq T(g)+T\left(
\left\vert f-g\right\vert \right)  ,$ that is, $T(f)-T(g)\leq T\left(
\left\vert f-g\right\vert \right)  $, and interchanging the role of $f$ and
$g$ we infer that $-\left(  T(f)-T(g)\right)  \leq T\left(  \left\vert
f-g\right\vert \right)  .$

If in addition $T(0)=0$ (for example, this happens when $T$ is monotone and
sublinear), then (\ref{eqmod}) yields the following inequality that
complements (\ref{eqh1}):
\begin{equation}
\left\vert T(f)\right\vert \leq T\left(  \left\vert f\right\vert \right)
\text{ for all }f\in E. \label{rh1}%
\end{equation}
This leads us to \emph{Holder's inequality for} $p=1$ \emph{and} $q=\infty:$%
\begin{equation}
\left\vert T(fg)\right\vert \leq T(\left\vert fg\right\vert )\leq T\left(
\left\vert f\right\vert \right)  \sup_{x\in X}|g(x)| \label{rh2}%
\end{equation}
for all $f,g\in E$ such that $fg\in E.$

If $X$ is a locally compact Hausdorff space and $T:C_{b}\left(  X\right)
\rightarrow\mathbb{R}$ is a positive linear functional for which $T(1)=1$,
then $T$ admits the integral representation $T(f)=\int_{X}fd\mu$ for a
suitable Borel probability measure $\mu$ and the difference%
\[
T(f^{2})-T(f)^{2}=\int_{X}f^{2}\mathrm{d}\mu-\left(  \int_{X}f\mathrm{d}%
\mu\right)  ^{2}%
\]
is just the \emph{variance} of $f$. The fact that the variance is nonnegative
follows from the Cauchy-Bunyakovsky-Schwarz inequality (the particular case of
H\"{o}lder's inequality for $p=q=2).$ Thus, in the general context of
sublinear and monotone operators $T:C_{b}(X)\rightarrow C_{b}(X)$, the
quantity%
\[
D_{T}^{2}(f)=T(1)\cdot T(f^{2})-T(f)^{2}%
\]
can be interpreted as the $T$-\emph{variance} of $f.$ The $T$-\emph{covariance
}of a pair of functions $f$ and $g$ in $C_{b}(X)$ can be introduced via the
formula
\[
\operatorname*{Cov}\nolimits_{T}(f,g)=T(1)\cdot T(fg)-T(f)T(g).
\]

\begin{problem}
\label{prob1}Under what conditions on $T$ is the following nonlinear version
of the Cauchy-Bunyakovsky-Schwarz inequality,%
\[
\left\vert \operatorname*{Cov}\nolimits_{T}(f,g)\right\vert \leq\sqrt
{D_{T}^{2}(f)}\sqrt{D_{T}^{2}(g)},
\]
true?
\end{problem}

Some results related to this problem are presented in what follows.

\begin{lemma}
\label{Lem1}If $T$ is a monotone and sublinear operator that maps $C_{b}(X)$
into itself, then
\[
D_{T}^{2}(-\left\vert f\right\vert ))=T(1)\cdot T(\left\vert f\right\vert
^{2})-|T(-\left\vert f\right\vert )|^{2}\geq0,
\]
for all $f\in C_{b}(X)$.
\end{lemma}

\begin{proof}
Since $T$ is monotone and subadditive, the fact that $0\leq(\lambda
-|f(x)|)^{2}$ for all $\lambda>0$ and $x\in X$ yields
\begin{equation}
0\leq T[(\lambda-|f|)^{2}](x)\leq\lambda^{2}T(1)(x)+2\lambda
T(-|f|)(x)+T(\left\vert f\right\vert ^{2})(x). \label{qpos}%
\end{equation}
Suppose by reductio ad absurdum that there exists $x_{0}\in X$ such that
\begin{equation}
|T(-|f|)(x_{0})|>\sqrt{T(1)(x_{0})\cdot T(f^{2})(x_{0})}. \label{detpos}%
\end{equation}
Then the second degree polynomial in $\lambda$,
\[
\lambda^{2}T(1)(x_{0})+2\lambda T(-|f|)(x_{0})+T(\left\vert f\right\vert
^{2})(x_{0})=0,
\]
will have two positive distinct solutions $\lambda_{1}<\lambda_{2}$. As a
consequence, for any $\lambda\in(\lambda_{1},\lambda_{2}),$
\[
\lambda^{2}T(1)(x_{0})+2\lambda T(-|f|\cdot|g|)(x_{0})+T(f^{2}g^{2}%
)(x_{0})<0,
\]
which contradicts condition (\ref{qpos}). Therefore (\ref{detpos}) does not
hold and the proof of Lemma 1 is done.
\end{proof}

The next lemma provides a partial answer to Problem \ref{prob1}.

\begin{lemma}
\label{Lem2}Suppose that $T:C_{b}(X)\rightarrow C_{b}(X)$ is a Choquet type
operator. Then for all pairs of functions $f,g\in C_{b}(X)$ such that
$\left\vert f\right\vert $ and $\left\vert g\right\vert $ are comonotone we
have the inequality%
\[
\left\vert \operatorname*{Cov}\nolimits_{T}(-\left\vert f\right\vert
,-\left\vert g\right\vert )\right\vert \leq\sqrt{D_{T}^{2}(-\left\vert
f\right\vert )}\sqrt{D_{T}^{2}(-\left\vert g\right\vert )}.
\]

\end{lemma}

\begin{proof}
Let $\lambda>0$ arbitrarily fixed. According to Lemma 1,%
\[
|T(-\left\vert f\right\vert -\lambda\left\vert g\right\vert )|^{2}\leq
T(1)\cdot T(\left\vert f\right\vert ^{2}+2\lambda\left\vert fg\right\vert
+\lambda^{2}\left\vert g\right\vert ^{2})
\]
while the fact that $T$ is comonotonic additive yields
\[
|T(-\left\vert f\right\vert -\lambda\left\vert g\right\vert )|^{2}=\left(
T(-\left\vert f\right\vert )+\lambda T(-\left\vert g\right\vert )\right)
^{2}.
\]
Therefore%
\[
\lambda^{2}D^{2}(-\left\vert g\right\vert )+2\lambda\left(  T(1)\cdot
T(\left\vert fg\right\vert )-T(-\left\vert f\right\vert )T(-\left\vert
g\right\vert )\right)  +D^{2}(-\left\vert f\right\vert )\geq0
\]
and taking into account that $\lambda>0$ was arbitrarily fixed one can
conclude (repeating the argument used in the proof of Lemma 1) that%
\[
\left\vert T(1)\cdot T(\left\vert fg\right\vert )-T(-\left\vert f\right\vert
)T(-\left\vert g\right\vert )\right\vert ^{2}\leq D_{T}^{2}(\left\vert
f\right\vert )D_{T}^{2}(\left\vert g\right\vert ).
\]

\end{proof}

\section{An application to Korovkin theory}

The following examples of Choquet type operators, borrowed from \cite{Gal-Mjm}%
, illustrate both our nonlinear extension of Korovkin's theorem\emph{ }stated
in\emph{ }Theorem \ref{thm1} and the nonlinear Cauchy-Bunyakovsky-Schwarz
inequalities stated in Lemma \ref{Lem1} and Lemma \ref{Lem2} :

- the \emph{Bernstein-Kantorovich-Choquet }operators $K_{n,\mu}%
:C([0,1])\rightarrow C([0,1]),$ defined by the formula
\[
K_{n,\mu}(f)(x)=\sum_{k=0}^{n}\frac{(C)\int_{k/(n+1)}^{(k+1)/(n+1)}%
f(t)\mathrm{d}\mu}{\mu([k/(n+1),(k+1)/(n+1)])}\cdot{\binom{n}{k}}%
x^{k}(1-x)^{n-k};
\]

- the \emph{Sz\'{a}sz-Mirakjan-Kantorovich-Choquet }operators\emph{ }%
$S_{n,\mu}:C([0,\infty))\rightarrow C([0,\infty)),$ defined by the formula
\emph{ }
\[
S_{n,\mu}(f)(x)=e^{-nx}\sum_{k=0}^{\infty}\frac{(C)\int_{k/n}^{(k+1)/n}%
f(t)\mathrm{d}\mu}{\mu([k/n,(k+1)/n])}\cdot\frac{(nx)^{k}}{k!};
\]

- the \emph{Baskakov-Kantorovich-Choquet }operators $V_{n,\mu}:C([0,\infty
))\rightarrow C([0,\infty))$ defined by the formula%
\[
V_{n,\mu}(f)(x)=\sum_{k=0}^{\infty}\frac{(C)\int_{k/n}^{(k+1)/n}%
f(t)\mathrm{d}\mu}{\mu([k/n,(k+1)/n])}\cdot{\binom{n+k-1}{k}}\frac{x^{k}%
}{(1+x)^{n+k}}.
\]

In the above examples $\mu$ is a submodular capacity whose restrictions to
suitable intervals are normalized by dividing the respective integrals by the
length of the interval of integration.

The aim of this section is to prove a quantitative estimate for the Korovkin
type result stated in Theorem \ref{thm1}. A basic ingredient is Lemma
\ref{Lem1}.

\begin{theorem}
\label{thm4} Let us consider the sequence of monotone, sublinear and
comonotone additive Bernstein-Kantorovich-Choquet operators $(K_{n,\nu})_{n}$
defined as above, but with $\nu$ a submodular normalized capacity satisfying
an inequality of the form $\nu\leq c\cdot\overline{\nu}$, with $c\geq1$. Then,
for all nonnegative functions $f\in C([0,1])$, all points $x\in\lbrack0,1]$
and all indices $n\in\mathbb{N}$, the following quantitative estimate holds:%
\begin{equation}
|K_{n,\nu}(f)(x)-f(x)|\leq(c+1)\omega_{1}(f;\sqrt{x^{2}+2xK_{n,\nu
}(-t)(x)+K_{n,\nu}(t^{2})(x)}), \label{eqthm4}%
\end{equation}
where $\omega_{1}(f;\delta)=\sup\{|f(t)-f(x)|:t,x\in\lbrack0,1],~|t-x|\leq
\delta)$ denotes the modulus of continuity.
\end{theorem}

\begin{proof}
For $x$ arbitrarily fixed, we have
\begin{align}
|K_{n,\nu}(f)(x)-f(x)|  &  =|K_{n,\nu}(f)(x)-K_{n,\nu}(f(x))(x)+K_{n,\nu
}(f(x)\cdot1)(x)-f(x)|\label{korovkin}\\
&  \leq|K_{n,\nu}(f(t)-f(x))(x)|+|f(x)|\cdot|K_{n,\nu}(1)(x)-1|\nonumber\\
&  \leq K_{n,\nu}(|f(t)-f(x)|)(x)+|f(x)|\cdot|K_{n,\nu}(1)(x)-1|,\nonumber
\end{align}
where the last inequality follows from the relation (\ref{rh1}).

On the other hand, from the properties of the modulus of continuity, for all
$t\in\lbrack0,1]$ and $\delta>0$, we have
\[
|f(t)-f(x)|\leq\omega_{1}(f;|t-x|)=\omega_{1}\left(  f;\delta\cdot\frac
{|t-x|}{\delta}\right)  \leq\left(  \frac{|t-x|}{\delta}+1\right)  \cdot
\omega_{1}(f;\delta).
\]
Choosing $\delta=|K_{n,\nu}(-|t-x|)(x)|=-K_{n,\nu}(-|t-x|)(x)$ (since
$K_{n,\nu}(-|t-x|)(x)\leq0$), we obtain
\[
|f(t)-f(x)|\leq\left(  \frac{|t-x|}{|K_{n,\nu}(-|t-x|)(x)|}+1\right)
\cdot\omega_{1}(f;|K_{n,\nu}(-|t-x|)(x)|).
\]
Applying to the last inequality the monotone and sublinear operator $K_{n,\nu
}$, we infer that
\begin{multline*}
K_{n,\nu}(|f(t)-f(x)|)(x)\\
\leq\left(  \frac{K_{n,\nu}(|t-x|)(x)}{|K_{n,\nu}(-|t-x|)(x)|}+K_{n,\nu
}(1)(x)\right)  \cdot\omega_{1}(f;|K_{n,\nu}(-|t-x|)(x)|).
\end{multline*}
Combining this fact with the inequality (\ref{korovkin}) we arrive at
\begin{align}
&  |K_{n,\nu}(f(t)-f(x))(x)|\label{eq44}\\
&  \leq\left(  \frac{K_{n,\nu}(|t-x|)(x)}{|K_{n,\nu}(-|t-x|)(x)|}+K_{n,\nu
}(1)(x)\right)  \cdot\omega_{1}(f;|K_{n,\nu}(-|t-x|)(x)|)\nonumber\\
&  +|f(x)|\cdot|K_{n,\nu}(1)(x)-1|.\nonumber
\end{align}

Denote $p_{n,k}(x)={\binom{n}{k}}x^{k}(1-x)^{n-k}$, to simplify the appearance
of formulas. Taking into account that $\nu\leq c\cdot\overline{\nu}$ we infer
from Remark \ref{rem1} $(d)$ and $(e)$ that
\begin{multline*}
K_{n,\nu}(|t-x|)(x)=\sum_{k=0}^{n}p_{n,k}(x)\cdot\frac{(C)\int_{k/(n+1)}%
^{(k+1)/(n+1)}|t-x|d\nu(t)}{\nu([k/(n+1),(k+1)/(n+1)])}\\
\leq c\sum_{k=0}^{n}p_{n,k}(x)\cdot\frac{(C)\int_{k/(n+1)}^{(k+1)/(n+1)}%
|t-x|d\overline{\nu}(t)}{\nu([k/(n+1),(k+1)/(n+1)])}\\
=c\sum_{k=0}^{n}p_{n,k}(x)\cdot\frac{|(C)\int_{k/(n+1)}^{(k+1)/(n+1)}%
-|t-x|d\nu(t)|}{\nu([k/(n+1),(k+1)/(n+1)])}=c\cdot|K_{n,\nu}(-|t-x|)(x)|,
\end{multline*}
which implies $K_{n,\nu}(|t-x|)(x)/|K_{n,\nu}(-|t-x|)(x)|\leq c$.

Now, since $K_{n,\nu}(1)=1$, the inequality stated by Lemma \ref{Lem1}, gives
us
\[
K_{n,\nu}(-|t-x|)(x)\leq\sqrt{K_{n,\nu}((t-x)^{2})(x)}\leq\sqrt{K_{n,\nu
}(t^{2})(x)+2xK_{n,\nu}(-t)(x)+x^{2}}.
\]
Replacing all these in (\ref{eq44}), we immediately obtain the inequality
(\ref{eqthm4}).
\end{proof}

\begin{remark}
\label{rem5} $(a)$ A concrete example of submodular normalized capacity
satisfying Theorem \emph{\ref{thm4}} is $\nu(A)=u\left(  \mathcal{L}%
(A)\right)  $, where $\mathcal{L}$ denotes the Lebesgue measure, $u$ is the
distortion defined by $u(t)=\frac{2t}{t+1}$ and $c=2$. Indeed, $\nu([0,1])=1$
and $\nu(A)=\frac{2\mathcal{L}(A)}{\mathcal{L}(A)+1}$. Denoting $\mathcal{L}%
(A)=x$, we get $\nu(A)=\frac{2x}{x+1}$ and
\[
\overline{\nu}(A)=1-\nu([0,1]\setminus A)=1-\frac{2\mathcal{L}([0,1]\setminus
A)}{\mathcal{L}([0,1]\setminus A)+1}=1-\frac{2(1-x)}{2-x}=\frac{x}{2-x}.
\]
Then, a simple computation shows that $\frac{2x}{x+1}\leq2\cdot\frac{2}{2-x}$
for all $x\in\lbrack0,1].$ Therefore Theorem \emph{\ref{thm4}} holds for $\nu$
when $c=2$.

$(b)$ Theorem \emph{\ref{thm4}} remains valid for submodular and normalized
capacities of the form $\nu(A)=u\left(  \mathcal{L}(A)\right)  $, with $u$ a
nondecreasing, concave function with $u(0)=0$, $u(1)=1$ and a constant
$c\geq1$ such that $u(x)\leq c[1-u(1-x)]$ for all $x\in\lbrack0,1]$.

$(c)$ Theorem \emph{\ref{thm4}} can be easily adapted to the case of
Sz\'{a}sz-Mirakjan-Kan\-to\-ro\-vi\-ch-Cho\-quet operators and
Baskakov-Kan\-to\-ro\-vich-Cho\-quet operators.
\end{remark}

\end{document}